\documentclass[a4,12pt]{amsart}
\usepackage{amssymb,amstext,amsmath,amscd,amsthm,amsfonts,enumerate,latexsym}
\usepackage{color}
\usepackage[dvipdfmx]{graphicx}
\theoremstyle{plain}
\newtheorem{thm}{Theorem}[section]
\newtheorem*{thm*}{Theorem}
\newtheorem*{cor*}{Corollary}

\newtheorem{prop}[thm]{Proposition}
\newtheorem{lem}[thm]{Lemma}
\newtheorem{cor}[thm]{Corollary}

\newtheorem*{claim*}{Claim}

\theoremstyle{definition}
\newtheorem{defn}[thm]{Definition}
\newtheorem{ex}[thm]{Example}
\newtheorem{rem}[thm]{Remark}

\newtheorem{conj}[thm]{Conjecture}
\newtheorem{quest}[thm]{Question}

\newtheorem*{acknowledgement}{Acknowledgement}

\theoremstyle{remark}

\numberwithin{equation}{thm}

\def\Assh{\operatorname{Assh}}

\def\Soc{\mathrm{Soc}}

\def\rank{\mathrm{rank}}

\def\e{\mathrm{e}}
\def\m{\mathfrak m}

\def\K{\mathrm{K}}

\newcommand{\rma}{\mathrm{a}}

\newcommand{\rme}{\mathrm{e}}

\newcommand{\rmK}{\mathrm{K}}

\newcommand{\calR}{\mathcal{R}}

\newcommand{\fka}{\mathfrak{a}}

\newcommand{\fkm}{\mathfrak{m}}

\newcommand{\fkM}{\mathfrak{M}}

\newcommand{\mapright}[1]{%
\smash{\mathop{%
\hbox to 1cm{\rightarrowfill}}\limits^{#1}}}

\newcommand{\mapleft}[1]{%
\smash{\mathop{%
\hbox to 1cm{\leftarrowfill}}\limits_{#1}}}

\def\ann{\operatorname{Ann}}

\def\height{\mathrm{ht}}

\def\Spec{\operatorname{Spec}}

\begin{document}

\setlength{\baselineskip}{14.5pt}
\title[Almost Gorenstein Rees algebras]{Almost Gorenstein  Rees algebras of $p_g$-ideals, good ideals, and powers of the maximal ideals}

\author{Shiro Goto}
\address{Department of Mathematics, School of Science and Technology, Meiji University, 1-1-1 Higashi-mita, Tama-ku, Kawasaki 214-8571, Japan}
\email{shirogoto@gmail.com}

\author{Naoyuki Matsuoka}
\address{Department of Mathematics, School of Science and Technology, Meiji University, 1-1-1 Higashi-mita, Tama-ku, Kawasaki 214-8571, Japan}
\email{naomatsu@meiji.ac.jp}

\author{Naoki Taniguchi}
\address{Department of Mathematics, School of Science and Technology, Meiji University, 1-1-1 Higashi-mita, Tama-ku, Kawasaki 214-8571, Japan}
\email{taniguchi@meiji.ac.jp}

\author{Ken-ichi Yoshida}
\address{Department of Mathematics, College of Humanities and Sciences, Nihon University, 3-25-40 Sakurajosui, Setagaya-Ku, Tokyo 156-8550, Japan}
\email{yoshida@math.chs.nihon-u.ac.jp}

\thanks{2010 {\em Mathematics Subject Classification.} 13H10, 13H15, 13A30}
\thanks{{\em Key words and phrases.} almost Gorenstein local ring, almost Gorenstein graded ring, regular local ring, Rees algebra, $p_g$-ideal, good ideal}

\begin{abstract}
Let $(A,\m)$ be a Cohen-Macaulay local ring and let $I$ be an ideal of $A$.
We prove that the Rees algebra $\calR(I)$ is an almost Gorenstein
ring in the following cases:
\begin{enumerate}
\item[(1)] $(A,\m)$ is a two-dimensional excellent Gorenstein normal domain over an algebraically closed field $K \cong A/\m$ and
$I$ is a $p_g$-ideal;
\item[(2)] $(A,\m)$ is a two-dimensional almost Gorenstein local ring having minimal multiplicity and $I=\m^{\ell}$ for all $\ell \ge 1$;
\item[(3)] $(A,\m)$ is a regular local ring of dimension
$d \ge 2$ and $I=\m^{d-1}$. Conversely, if  $\calR(\m^{\ell})$ is an almost Gorenstein graded ring for some $\ell \ge 2$ and $d \ge 3$, then $\ell=d-1$.
\end{enumerate}
\end{abstract}

\maketitle
\section{Introduction}\label{intro}
In \cite{GMTY2}, the authors proved that
for any $\m$-primary integrally closed ideal $I$
in a two-dimensional regular local ring $(A,\m)$,
its Rees algebra $\calR(I)$ is an almost Gorenstein graded ring.
As a direct consequence, we have that the Rees algebra
$\calR(\m^{\ell})$ is an almost Gorenstein graded ring
for every integer $\ell \ge 1$.
The main purpose of this paper is to extend these results
to other classes of rings and ideals.

The notion of almost Gorenstein rings in our sense was
introduced by Barucci and Fr\"oberg \cite{BF},
where they dealt with one-dimensional analytically unramified local rings.
Goto, Matsuoka, and Phuong  \cite{GMP} extended
the notion to arbitrary (but still of dimension one) Cohen-Macaulay local rings.
Goto, Takahashi, and Taniguchi \cite{GTT}
gave the definition of
almost Gorenstein graded/local rings for higher-dimensional cases.

Let us recall the definition of almost Gorenstein rings.

\begin{defn}[{Goto et al.\ \cite[Definition 3.3]{GTT}}]
\label{AGlocal-def} Let $(A,\m)$ be a Cohen-Macaulay local ring that possesses the canonical module $\K_A$.
Then $A$ is said to be an \textit{almost Gorenstein local ring}, if there exists an exact sequence
\[
0 \to A \to \mathrm{K}_A \to C \to 0
\]
of $A$-modules such that $\mu_A(C) = \e_\m^0(C)$,
where $\mu_A(C)$ denotes the number of elements in a minimal system of generators of $C$ and $\e_\m^0(C)$
is the multiplicity of $C$ with respect to $\m$.
Note that such an $A$-module $C$ is called an {\it Ulrich}
$A$-module; see e.g.\ \cite[Section 2]{GTT}.
\end{defn}

Let $R = \bigoplus_{n \ge 0}R_n$ be a Cohen-Macaulay graded ring such that  $R_0=A$ is a local ring.
Suppose that $R$ possesses the graded canonical module $\K_R$.
Then $\rma (R)=-\min\{n \in \mathbb{Z} \mid  [\K_R]_n \ne 0\}$ is called the \textit{$\rma$-invariant} of $R$; see
e.g.\ \cite[Definition 3.14]{GW}.

\begin{defn}[{Goto et al.\ \cite[Definition 8.1]{GTT}}]
\label{AGgraded-def}
Let $R$ be as above.
Let $\fkM$ be the unique graded maximal ideal of $R$
and set $a=\rma(R)$.
Then $R$ is said to be an \textit{almost Gorenstein graded ring}, if there exists an exact sequence
\[
0 \to R \to \mathrm{K}_R(-a) \to C \to 0
\]
of graded $R$-modules such that  $\mu_R(C) = \e_\fkM^0(C)$, where $\mu_R(C)$ denotes the number of elements in a minimal system of generators of $C$ and $\e_\fkM^0(C)$ is the multiplicity of $C$ with respect to $\fkM$.
Here $\K_R(-a)$ denotes the graded $R$-module whose underlying $R$-module is the same as that of $\K_R$ whose grading is given by $[\K_R(-a)]_n = [\K_R]_{n-a}$ for all $n \in \mathbb{Z}$.
\end{defn}

Note that the local ring $R_{\fkM}$ is almost Gorenstein if $R$ is an almost Gorenstein graded ring, because $C_{\fkM}$ is an Ulrich $R_{\fkM}$-module and $\mathrm{K}_{R_{\fkM}} \cong [\mathrm{K}_R]_{\fkM}$. Unfortunately, the converse is not true in general (see e.g.\ \cite[Theorems 2.7 and 2.8]{GMTY2} and \cite[Example 8.8]{GTT}).

Any Gorenstein local ring is an almost Gorenstein local ring.
Any rational singularity in dimension two is an
almost Gorenstein local ring (see \cite[Section 11]{GTT}).
All known examples of
Cohen-Macaulay local rings of finite representation
type are almost Gorenstein local rings
(see \cite[Section 12]{GTT}).
Moreover, a numerical semigroup ring $k[[H]]$ is
an almost Gorenstein ring if and only if $H$ is an almost
symmetric semigroup (see \cite{BF}).
Note that the notion of almost Gorenstein rings
in our sense is different from that in \cite{HV}.

Moreover, the following results are known
as examples of
higher-dimensional almost Gorenstein rings:
for a parameter ideal $Q$ in a regular local ring $A$ of
dimension $d \ge 3$:
\begin{enumerate}
\item[(1)] the Rees algebra
$\calR(Q)=\bigoplus_{n \ge 0} Q^n$ is an almost Gorenstein
graded ring if and only if $Q=\m$ (see \cite[Theorem 8.3]{GTT});
\item[(2)] $\calR(Q)_{\fkM}$
is always an almost Gorenstein local ring, where
$\fkM$ denotes the unique graded maximal ideal of $\calR(Q)$ (see \cite{GMTY1}).
\end{enumerate}

The main results in this paper are the following theorems,
which are extensions of the main result in \cite{GMTY2}.
Note that any $\m$-primary integrally closed ideal $I$ in a two-dimensional
excellent regular local ring $A$ over
an algebraically closed field satisfies
the assumption in the following theorem.
Thus, this theorem essentially extends
\cite[Theorem 1.3]{GMTY2}.

Now let $(A,\m)$ be a two-dimensional excellent
normal local ring over an algebraically closed field.
For an $\m$-primary ideal $I \subset A$,
$I$ is called a {\it $p_g$-ideal} (see \cite{OWY1}),
if the Rees algebra
$\calR(I)$ is a Cohen-Macaulay normal domain;
see the following section for the definition and basic
properties of $p_g$-ideals.

\begin{thm}[See Theorem \ref{Sect2-thm}]
\label{Main-PGRees}
Let $(A,\m)$ be a two-dimensional excellent Gorenstein
normal local ring over an algebraically closed field
and let $I \subset A$ be a $p_g$-ideal.
Then the Rees algebra $\calR(I)$ of $I$ is an almost Gorenstein graded ring.
\end{thm}

Since any two-dimensional rational singularity is an almost Gorenstein local ring
having minimal multiplicity, the following theorem provides
a large class of local rings for which the Rees algebras of all powers of the maximal ideal are almost Gorenstein graded rings.

\begin{thm}[See Theorem \ref{Good-AGG}]
 \label{Main-good}
Let $(A,\m)$ be a two-dimensional almost Gorenstein local ring
having minimal multiplicity.
Then $\calR(\m^{\ell})$ is an almost Gorenstein graded ring for every $\ell \ge 1$.
\end{thm}

\begin{cor} \label{2-dimRatpower}
Let $(A,\m)$ be a two-dimensional rational singularity.
Then $\calR(\m^{\ell})$ is an almost Gorenstein graded ring for every $\ell \ge 1$.
\end{cor}

The following theorem is a higher-dimensional analog of \cite[Corollary 1.4]{GMTY2}.
We note that if $d=5$ and $\ell=2$, then
$\calR(\m^{2})_{\fkM}$ is an almost Gorenstein local ring
but $\calR(\m^2)$ is not an almost Gorenstein graded ring.

\begin{thm}[See Proposition \ref{AG-local} and
Theorem \ref{NotGraded}]
 \label{Main-RLRpower}
Let $(A,\m)$ be a regular local ring of
dimension $d \ge 2$ that possesses an infinite residue class field.
Then$:$
\begin{enumerate}
 \item[$(1)$] $\calR(\m^{\ell})$ is an almost Gorenstein graded ring if and only if
$\ell=1$, $d=2$ or $\ell=d-1;$
 \item[$(2)$] for $\ell \ge 2$ and $d \ge 3$,
$\calR(\m^{\ell})_{\fkM}$ is an almost Gorenstein local ring if and only if $\ell\mid  d-1$, where
$\fkM$ denotes the graded maximal ideal of $\calR(\m^{\ell})$.
\end{enumerate}
\end{thm}

Note that under the above assumption, the associated
graded ring $G(\m^{\ell})$ is Gorenstein if and only if
$\ell\mid  d-1$; see \cite[Theorem 2.4]{Hy}.

We now briefly explain how this paper is organized.
The proof of Theorem \ref{Main-PGRees} is given in
 Section 2.
In Section 3, we prove Theorem \ref{Main-good}.
In Section 4, we prove Theorem \ref{Main-RLRpower}.

In what follows, unless otherwise specified, let $(A,\m)$ be a Cohen-Macaulay local ring.  Let $\mathrm{K}_A$ denote the canonical module of $A$.
For each finitely generated $A$-module $M$, let $\mu_R(M)$ (respectively $\e_\m^0(M)$)
denote the number of elements in a minimal system of generators for $M$
(respectively the multiplicity of $M$ with respect to $\m$).

\section{Rees algebras of $p_g$-ideals
(Proof of Theorem $\ref{Main-PGRees}$)}

The purpose of this section is to prove
Theorem \ref{Main-PGRees}.
Throughout this section, let $(A,\m)$ be a Cohen-Macaulay local ring of dimension two,
and suppose that $A$ is generically a Gorenstein ring and that
it possesses a canonical ideal $\K_A \subset A$.
Moreover, let $I \subset A$ be an $\m$-primary ideal, and let $Q$ be its minimal reduction of $I$.
Suppose that $I$ is stable, that is, $I^2=QI$ and $I \ne Q$, and
set $J=Q:I$.
Then the Rees algebra
\[
\calR= \calR(I) = A[It] \subseteq A[t]
\]
of $I$ is a Cohen-Macaulay ring \cite{GS} with $\rma(\calR) = -1$, where $t$ denotes an indeterminate over $A$. We denote by $\fkM = \m \calR + \calR_+$ the graded maximal ideal of $\calR$.

\subsection{$P_g$-ideals}
Assume that $A$ is an excellent normal local domain over an algebraically closed field and
there exists a resolution of singularities
$f \colon X \to \Spec A$ with $E=f^{-1}(\m)=\bigcup_{i=1}^r E_i$.
Then $p_g(A)=\ell_A(H^1(\mathcal{O}_X))$ is called the \textit{geometric genus} of $A$,
which is independent of the choice of the resolution of singularities.
Recall that any $\m$-primary integrally closed ideal $I$ can be written
as $I= H^0(\mathcal{O}_X(-Z))$ for some resolution of singularities
$X \to \Spec A$ and some anti-nef cycle $Z$ on $X$
so that $I\mathcal{O}_X=\mathcal{O}_X(-Z)$; see e.g.\ \cite[Section 18]{Li}.
Then $I$ is said to be \textit{represented on $X$ by $Z$}.

Now let us recall the notion of $p_g$-ideals.
Fix a resolution of singularities $X \to \Spec A$.
Let $Z$ be an anti-nef cycle on $X$, and assume that $\mathcal{O}_X(-Z)$ has no fixed component,
that is, $H^0(\mathcal{O}_X(-Z)) \ne H^0(\mathcal{O}_X(-Z-E_i))$ for every $E_i \subset E$.
Then we have
\[
\ell_A(H^1(\mathcal{O}_X(-Z))\le p_g(A).
\]
If equality holds true, then $\mathcal{O}_X(-Z)$ is generated by global sections
(see \cite[Theorem 3.1]{OWY1}) and $Z$ is
called a \textit{$p_g$-cycle}.

\begin{defn}[\textrm{Okuma et al.\ \cite[Definition 3.2]{OWY1}}]
\label{defn-pgideal}
An $\m$-primary ideal $I \subset A$ is called a
\textit{$p_g$-ideal}
if it is represented  by a $p_g$-cycle $Z$ on some resolution of singularities $X \to \Spec A$.
\end{defn}

The following criterion for $p_g$-ideals is very useful.

\begin{lem}[\textrm{Okuma et al.\ \cite[Theorem 1.1]{OWY2}}]
\label{Chara-pg}
Let $I \subset A$ be an $\m$-primary ideal that is not a
parameter ideal.
Then the following conditions are equivalent:
\begin{enumerate}
\item[$(1)$] $I$ is a $p_g$-ideal;
\item[$(2)$]
$I^2=QI$ for some parameter ideal $Q \subset I$
and $I^n$ is integrally closed for every $n \ge 1$;
\item[$(3)$] $\mathcal{R}(I)$ is a Cohen-Macaulay normal domain.
\end{enumerate}
\end{lem}

Now let us recall the notion of rational singularities
(of dimension two).
The ring $A$
is called a  \textit{rational singularity} if $p_g(A)=0$.
For instance, toric singularities and quotient singularities are
typical examples of rational singularities.
The notion of $p_g$-ideals can be regarded as analogous to the notion of integrally closed ideals in a rational singularity.
In fact, Lipman \cite[Theorem 12.1]{Li} showed that
any $\m$-primary integrally closed ideal of $A$
is a $p_g$-ideal provided that $A$ is a rational singularity.

It is known that the Rees algebra $\mathcal{R}(I)$ of an
integrally closed ideal $I$ in a rational singularity
possesses some good property (e.g.\ rationality) by Lipman \cite{Li}.
Moreover, Lemma \ref{Chara-pg} implies that
$\mathcal{R}(I)$ is a Cohen-Macaulay normal domain if $I$ is a $p_g$-ideal of $A$.
We pose the following conjecture.

\begin{conj} \label{pg-conj}
If $I \subset A$ is a $p_g$-ideal, then
$\mathcal{R}(I)$ an almost Gorenstein graded ring.
\end{conj}

As a partial answer, we prove the following theorem, which implies that Theorem \ref{Main-PGRees} holds.

\begin{thm} \label{Sect2-thm}
Assume that $A$ is Gorenstein under the assumption in this section.
Let $I \subset A$ be a $p_g$-ideal and $Q$ its minimal reduction.
Put $J=Q\colon I$ and $\calR=\calR(I)$.
Then there exists a short exact sequence and elements $f \in \m$, $g \in I$, and $h \in J$ such that
\[
 0 \to \calR \stackrel{\varphi}{\longrightarrow} \K_{\calR}(1)\cong J\calR \to C \to 0,
\]
where $\varphi(1) = h$ and $\fkM C = (f,gt)C$.
In particular, $\calR$ is an almost Gorenstein graded ring.
\end{thm}

\begin{cor} \label{cor-PGRees}
Assume that $A$ is a rational double point, that is,
it is a Gorenstein rational singularity that is not regular.
Then $\calR(I)$ is an almost Gorenstein normal
graded ring for any $\m$-primary integrally closed ideal $I \subset A$.
\end{cor}

\begin{proof}
Under this assumption, any $\m$-primary
integrally closed ideal is a $p_g$-ideal but not
a parameter ideal. Thus, we can apply Theorem~\ref{Sect2-thm}.
\end{proof}

Since any regular local ring is a rational singularity,
we have the following.

\begin{ex} \label{ex-regular}
Let $A$ be a regular local ring with $\dim A=2$.
Then $\mathcal{R}(I)$ is an almost Gorenstein graded ring
for any integrally closed ideal
$I \subset A$.
In particular, $\calR(\m^{\ell})$ is an almost
Gorenstein graded ring for every $\ell \ge 1$.
\end{ex}

\begin{rem}
Suppose that $A$ is a Gorenstein local ring of dimension two and $I=(a,b)$ is a parameter ideal of $A$.
Then $\calR(I) \cong A[X,Y]/(aY-bX)$ is a Gorenstein ring.
\end{rem}

Okuma et al.\ \cite[Theorem 1.2]{OWY1} showed that any excellent normal
local domain of dimension two admits a $p_g$-ideal.
Therefore, the following corollary is obtained from Theorem \ref{Main-PGRees}.

\begin{cor} \label{ex-AGRees-existence}
For any excellent normal Gorenstein local domain of dimension two
over an algebraically closed field $k$, there exists an $\m$-primary
ideal $I$ so that $\mathcal{R}(I)$ is an almost Gorenstein graded ring.
\end{cor}

\begin{ex}[\textrm{Okuma et al.\ \cite{OWY2}}]
Let $p \ge 1$ be an integer.
\begin{enumerate}
\item Let $A=k[[x,y,z]]/(x^2+y^3+z^{6p+1})$.
Then $I_k=(x,y,z^k)$ is a $p_g$-ideal
for every $k=1,2,\ldots, 3p$.
\item  Let $A=k[[x,y,z]]/(x^2+y^4+z^{4p+1})$.
Then $I_k=(x,y,z^k)$ is a $p_g$-ideal
for every
$k=2,\ldots,2p$. However, $I_1=\m$ is not.
\end{enumerate}
When this is the case, $p_g(A)=p$.
\end{ex}
\subsection{Proof of Theorem $\ref{Main-PGRees}$}
The purpose of this subsection is to prove Theorem \ref{Main-PGRees}.
In what follows, we always assume that the assumption of Theorem \ref{Main-PGRees} holds.
The following lemma, which is a special case of \cite[Lemma 5.1]{GI},
 plays a key role in the proof. Note that if $A$ is Gorenstein, then it is a special case of \cite[Theorem 2.7(a)]{U}.

\begin{lem} \label{Key}
Suppose that $A$ possesses the canonical ideal $\K =\K_A$. Then $\K_{\mathcal{R}}(1) \cong (Q\K \colon_{\K} I)\mathcal{R}$ as graded $\mathcal{R}$-modules.
\end{lem}

\begin{proof}
Since $I^2=QI$, if we put $\omega_i=Q^i\K \colon_{\K} I$, then
$\omega_i = I^{i-1}\K=Q^{i-1}\K$ for every $i \ge 1$ and
$\K_{\mathcal{R}}(1)$ is isomorphic to
 $\omega_1 \mathcal{R}=(Q\K\colon_{\K} I)\mathcal{R}$.
\end{proof}

The following two lemmata
play important roles in the proof of
the main theorem.

\begin{lem}[\textrm{Okuma et al.\ \cite[Theorem 3.5]{OWY2}}]
\label{jointred_pg}
Assume that $I$ is a $p_g$-ideal and
$J$ is an integrally closed $\m$-primary ideal.
Then there exist $a \in I$, $b \in J$ such that
$IJ=aJ+bI$.
\end{lem}

\begin{lem}[\textrm{Okuma et al.\ \cite{OWY3}}]
\label{colon_pg}
Assume that $I \subset A$ is a $p_g$-ideal.
If $Q$ is a minimal reduction of $I$, then
$J=Q\colon I$ is also a $p_g$-ideal.
\end{lem}

We are now ready to prove
Theorem \ref{Sect2-thm}.

\begin{proof}[Proof of Theorem \ref{Sect2-thm}]
Assume that $A$ is Gorenstein and $I$ is a $p_g$-ideal.
Then $J=Q\colon I$ is also a $p_g$-ideal by
Lemma \ref{colon_pg}.
Hence, Lemma \ref{jointred_pg} implies that there exist
$f \in \m$, $g \in I$, and $h \in J$
such that
\[
IJ=gJ+Ih, \qquad \m J=fJ+\m h
\]
because $I, J$ are $p_g$-ideals and $\m$ is integrally closed
(see also \cite{V}).

\medskip

\begin{flushleft}
{\it Claim.} We claim that
$ \mathfrak{M}\cdot J\mathcal{R}
\subset (f,gt)J\mathcal{R} + \mathcal{R}h$.
\end{flushleft}

\medskip

In fact,
\begin{eqnarray*}
[\mathfrak{M}\cdot J\mathcal{R}]_0
&=&
\mathfrak{m}J=fJ + \mathfrak{m}h \subseteq fJ+Ah =
[(f,gt)J\mathcal{R} + \mathcal{R}h]_0, \\[1mm]
[\mathfrak{M}\cdot J\mathcal{R}]_1
&=&
IJ+\mathfrak{m}JI
=IJ=gJ+Ih \subseteq [(f,gt)J\mathcal{R} + \mathcal{R}h]_1,
 \\[1mm]
[\mathfrak{M}\cdot J\mathcal{R}]_n &=&
I^nJ=(gJ)I^{n-1}+I^{n-1}Ih \subseteq [(f,gt)J \mathcal{R} + \subseteq \mathcal{R}h]_{n} \quad
\text{for all}\ n \ge 2.
\end{eqnarray*}
Thus, we have proved the claim.

Since $\K_{\mathcal{R}}(1) \cong J\mathcal{R}$ by Lemma \ref{Key} and
$\rma(\mathcal{R})=-1$,
if we define an $\mathcal{R}$-linear map
$\varphi$ by $\varphi(1)=h$, then we have an exact sequence
\[
\mathcal{R} \stackrel{\varphi}{\to}
\K_{\mathcal{R}}(1) = J \mathcal{R} \to C \to 0
\]
so that $C/\mathfrak{M}C = C/(f,gt)C$.
Hence, $C$ is an Ulrich $A$-module by \cite[Lemma 3.1]{GTT}.
As $\dim C_{\mathfrak{M}}\le 2<\dim \mathcal{R}=3$,
$\varphi_{\mathfrak{M}}$ is injective by \cite[Lemma 3.1]{GTT} again.
This means $\varphi$ is injective.
Therefore, we conclude that $\mathcal{R}$ is an almost Gorenstein graded ring.
\end{proof}

\section{Rees algebras of good ideals}

We first recall the notion of good ideals, which was introduced in \cite{GIW}.

\begin{defn} \label{good-def}
Let $I \subset A$ be an $\m$-primary ideal.
Then $I$ is called a \textit{good} ideal if
$I^2=QI$ and $Q\colon I = I$ for some minimal reduction $Q$ of $I$.
\end{defn}

Now assume that $\m^2=Q\m$ for some minimal reduction $Q$ of $\m$, that is,
$A$ has minimal multiplicity.
Then $\m$ is an integrally closed good ideal.
For instance, if $A$ is a two-dimensional
rational singularity, then $\m$ is a good ideal.

The following proposition says that the definition of good ideals is independent of the choice of its minimal reduction.

\begin{prop} \label{Good-chara}
Put $\calR=\calR(I)$ and
$G=G(I)=\oplus_{n \ge 0} I^n/I^{n+1}$.
Then the following conditions are equivalent:
\begin{enumerate}
 \item[$(1)$] $I$ is a good ideal;
 \item[$(2)$] $G$ is a Cohen-Macaulay ring with
$\rma(G)=1-d$ and $\Soc(H_{\fkM}^d(G)) \subset [H_{\fkM}^d(G)]_{1-d}$.
\end{enumerate}
\end{prop}

\begin{proof}
$(1) \Longrightarrow (2)$
By definition, there exists a minimal reduction $Q$ of $I$
so that $I^2=QI$ and $Q\colon I = I$.
In particular, $G$ is a Cohen-Macaulay ring
with $\rma(G)=1-d$.
Write $Q=(a_1,\ldots,a_d)$ and
$a_i^{*}:=a_i+I^2 \in [G]_1$ for each $i=1,2,\ldots,d$.
Since $a_1^{*},a_2^{*},\ldots,a_d^{*}$ forms a regular sequence in $G$, we have
\[
G/(a_1^{*},\ldots,a_d^{*})G \cong G(I/Q) \cong
A/I \oplus I/Q=:\overline{G}.
\]
Then
\[
H_{\fkM}^d(G) \cong H_{\fkM}^{d-1}(G/a_1^{*}G)(1)
\cong \cdots
\cong H_{\fkM}^0(G/(a_1^{*},\ldots,a_d^{*})G)(d)
\cong \overline{G}(d).
\]
Thus, it suffices to show $\Soc(\overline{G}) \subset [\overline{G}]_1$.
Now suppose that $\alpha =(x+I,y+Q) \in \Soc(\overline{I})$, where $x \in A$ and $y \in I$.
By definition, we have $(z+I)\alpha =0$
for any $z \in I$.
That is, $zx \in Q$ for any $z \in I$.
Hence, $x \in Q\colon I = I$ because $I$ is good.
Therefore, $\alpha =(0+I,y+Q) \in [\overline{G}]_1$, as
required.

\medskip

\noindent
$(2) \Longrightarrow (1)$
As $G$ is a Cohen-Macaulay ring with $\rma(G)=1-d$, we have
$I^2=QI$ for some minimal reduction $Q$ of $I$.
Write $Q= (a_1,\ldots,a_d)$ and $a_i^{*}:=a_i+I^2 \in [G]_1$ for each $i=1,2,\ldots,d$.
Then $H_{\fkM}^d(G) \cong \overline{G}(d)$, where
$\overline{G}=A/I\oplus I/Q$.

Now suppose that $I$ is not good, that is, $I \subsetneq
Q\colon I$.
Then we can choose $x \in I \colon \m \setminus I$ so that
$x \in Q \colon I$.
If we put $\alpha = x+I \in [\overline{G}]_0$, then
$\beta \alpha =0$ in $\overline{G}$ for every
$\beta \in \m \overline{G}+[\overline{G}]_1$, the maximal ideal of $\overline{G}$.
That is, $0 \ne \alpha \in [\Soc(\overline{G})]_0$.
However, this contradicts the assumption.
Therefore, $I=Q\colon I$ and $I$ is a good ideal.
\end{proof}

In what follows, we consider the two-dimensional case.
As a corollary of Proposition \ref{Good-chara},
we can compute $\K_{\calR(I)}$ for a good ideal $I$.

\begin{cor} \label{good-cano}
Assume that $\dim A=2$, $I \subset A$ is a good ideal, and $A$ possesses the canonical ideal $\K =\K_A$. Set $\calR=\calR(I)$ and $G=G(I)$.
Then $\K_{\calR}(1) \cong I\K \calR$.
\end{cor}

\begin{proof}
Set
$\omega_{-1}=\omega_0 = \K$, $\omega_1=Q\K\colon_{\K} I$ and
$\omega_i=I^{i-1} \omega_1$ for each $i \ge 2$.
Then $[\K_G]_i = \omega_{i-1}/\omega_i$ for every
$i \ge 1$; see \cite[Theorem 2.1]{GI}.
By Proposition \ref{Good-chara} and duality,
$\K_G$ is generated by $[\K_G]_1$ as an $G$-module.
Hence, $\omega_1/\omega_2 =[\K_G]_2
\subset G \cdot \omega_0/\omega_1$.
That is, $\omega_1 = I\omega_0 + \omega_2 = I\K+ I\omega_1$.
This implies that $\omega_1 = I\K$ by Nakayama's lemma.
It follows from Lemma \ref{Key} that
$\K_{\calR}(1) \cong I\K \calR$, as required.
\end{proof}

In this section, we consider the following question.

\begin{quest} \label{good-Rees}
Assume that $I$ is a good ideal.
When is $\calR(I)$ an almost Gorenstein graded ring?
\end{quest}

The following theorem is the main result in this section,
which gives a partial answer to the question above.

\begin{thm}
\label{Good-AGG}
Let $(A,\m)$ be a two-dimensional almost Gorenstein local ring. Assume that $\m$ is good.
Then $\calR(\m^{\ell})$ is an almost Gorenstein graded ring
for every $\ell \ge 1$.
\end{thm}

\begin{proof}
Set $\calR=\calR(\m^{\ell})$. By \cite[Remark 3.2]{GTT}, the ring $A$ contains the canonical ideal $\K=\K_A$. Fix $\ell \ge 1$ a positive integer. Note that $\m^{\ell}$ is a good ideal, so that
$\K_{\calR}(1) \cong \m^{\ell} \K \calR$ by Corollary \ref{good-cano}.
Then it is enough to prove the following claim.

\medskip

\begin{flushleft}
{\it Claim.}
There exists $f \in \m$,
$g \in \m^{\ell}$ and $h \in \m^{\ell} \K$ such that
\[
\m^{\ell+1} \K = f \m^{\ell}\K + \m h,\qquad
\m^{2\ell}\K = g \m^{\ell} \K + \m^{\ell} h.
\]
\end{flushleft}

\medskip

Note that this gives a proof of the theorem.
Indeed, by a similar argument as in the proof of
Theorem \ref{Sect2-thm}, we have
$\fkM \cdot \K_{\calR} \subset (f,gt)\K_{\calR} + \calR h$.
This yields a graded short exact sequence
\[
 0 \to \calR \stackrel{\psi}{\to} \K_{\calR}(1) \cong \m \K \calR \to C
\to 0,
\]
where $\psi(1) = h$ and $\fkM C = (f,gt)C$.
Namely, $\calR$ is an almost Gorenstein graded ring.

Let us prove the above claim.
First suppose that $A$ is Gorenstein.
Then $\K=A$.
By assumption, we can take $y,z \in \m$ so that
$\m^2=(y,z)\m$.
If we set $f=y$, $g=y^{\ell}$, and $h=z^{\ell}$, then
\[
\m^{\ell+1}=f \m^{\ell}+\m h, \qquad
\m^{2\ell} = g\m^{\ell} + \m^{\ell} h,
\]
as required.

Next suppose that $A$ is not Gorenstein. Then
we have a short exact sequence
\[
0 \to A \stackrel{\varphi}{\to} \K=\K_A \to C \to 0
\]
such that $C$ is an Ulrich $A$-module of $\dim C =1$.
If we put $x=\varphi(1)$, then
$x \in \K \setminus \m \K$ by
\cite[Corollary 3.10]{GTT}.
Choose $y,z \in \m$ so that:
\begin{enumerate}
 \item[(i)] $(y,z)$ is a minimal reduction of $\m$;
 \item[(ii)] the  image of  $(y)$ in  $A/\ann_A C$ is a minimal reduction of $\m/\ann_A C$.
\end{enumerate}
From (i), we have $\m^2=(y,z)\m$ and, thus,
$\m^{\ell}=(y,z)^{\ell-1}\m$ and $\m^{2\ell}=(y^{\ell},z^{\ell})\m^{\ell}$.
Then (ii) implies $\m C= yC$. In particular,
$\m \K \subset y\K +xA$.
Hence, $\m \K=y\K+x\m$ because $x \notin \m \K$.
By multiplying this by $\m^{\ell}$, we obtain
\begin{eqnarray*}
\m^{\ell+1}\K
& = & y \m^{\ell} \K + xz \m^{\ell} \\
& = & y \m^{\ell} \K + xz (y,z)^{\ell-1}\m \\
& = & y \m^{\ell} \K + xz^{\ell} \m.
\end{eqnarray*}
Moreover, we have $\m^{2\ell}\K=y^{\ell}\m^{\ell}\K+z^{\ell}\m^{\ell}\K$.
On the other hand,
\begin{eqnarray*}
z^{\ell}\m^{\ell}\K
& = & z^{\ell} \m^{\ell-1}\m \K \\
& = & z^{\ell} \m^{\ell-1} (y\K+x\m) \\
& \equiv & yz^{\ell} \m^{\ell-1}\K
\pmod{xz^{\ell}\m^{\ell}} \\
& \equiv & yz^{\ell} \m^{\ell-2} (y\K+x\m)
\pmod{xz^{\ell}\m^{\ell}} \\
&\equiv & y^2 z^{\ell} \m^{\ell-2}
K\pmod{xz^{\ell}\m^{\ell}} \\
& \equiv & \cdots \equiv y^{\ell}z^{\ell} \K \pmod{xz^{\ell}\m^{\ell}}.
\end{eqnarray*}
Hence, $\m^{2\ell}\K=y^{\ell}z^{\ell}\K + xz^{\ell}\m^{\ell}$.
If we set $f = y$, $g=y^{\ell} \in \m^{\ell}$ and $h=xz^{\ell} \in \m^{\ell}\K$,
then we obtain the required
equality.
\end{proof}

Let us explore the example to show how Theorem \ref{Good-AGG} works.

\begin{ex} \label{Veronese}
Let $r \ge 2$ be an integer.
Let $A=k[[s^r,s^{r-1}t,\ldots,st^{r-1},t^{r}]]$ be
the $r$th Veronese subring of $k[[s,t]]$.
Note that $(A,\m)$ is a rational singularity.
Set
\[
\K=(s^{r-1}t,s^{r-2}t^2,\ldots,st^{r-1}).
\]
If we take
\[
x=st^{r-1} \in \K, \quad y=s^r, \quad z=t^r,
\]
then $\m \K = y \m \K + x \m$ and $\m^2=(y,z)\m$.
Thus, $\mathcal{R}(\m)$ is an almost Gorenstein graded ring.
\end{ex}

If $A$ is a Cohen-Macaulay local ring of $\rme_{\m}^0(A)=2$, then
it has minimal multiplicity and it is Gorenstein, that is, $\K=A$.

\begin{ex} \label{mult2}
Suppose that $(A,\m)$ is a Cohen-Macaulay local ring of $\rme_{\m}^0(A)=2$.
Then $\mathcal{R}(\m)$ is an almost Gorenstein graded ring.
\end{ex}

In the rest of this section, we consider Question \ref{good-Rees} in the higher-dimensional case.
To prove our result, we need the following lemma, which
is very useful in proving the almost Gorensteinness of the Rees
algebra.

\begin{lem} \label{useful}
Let $(A,\m)$ be a regular local ring and let $I \subsetneq A$ be an ideal of positive height.
Set $\calR=\calR(I)$.
If $\calR_{\fkM}$ is a Cohen-Macaulay local ring with
$\K_{\calR} = \sum_{i=1}^{c+1} \calR t^i$ for some $c \ge 0$,
then $\calR_{\fkM}$ is an almost Gorenstein local ring.
\end{lem}

\begin{proof}
Note that
\[
\K_{\calR} \cong At + At^2+\cdots + At^{c}+\calR t^{c+1}.
\]
Set $C=\K_{\calR}/\calR t^{c+1}$.
Then $C \cong A\overline{t}+A\overline{t^2} + \cdots + A\overline{t^c}  \cong A^{\oplus c}$ as $A$-modules.
If we set $\fka=\calR_{+}^c$ and $\overline{\calR}=\calR/\fka$, then $C$ is an $\overline{R}$-module because $\fka C=0$.
Moreover, $A \hookrightarrow \calR \to \overline{\calR}$
is a finite morphism and $\m \overline{\calR}$ is a reduction of $\overline{\fkM} = \fkM/\fka$, we have
\[
\rme_{\fkM}^0(C)
=\rme_{\overline{\fkM}}^0(C)
=\rme_{\m \overline{\calR}}^0(C)
=\rme_{\m}^0(C)=\rme_{\m}^0(A) \cdot \rank_A C=1 \cdot c=c.
\]
On the other hand, since $\mu_{\calR}(C)=c$, $C_{\fkM}$ is an
Ulrich $\calR_{\fkM}$-module and, thus, $\calR_{\fkM}$ is an
almost Gorenstein local ring.
\end{proof}

The following theorem gives a complete answer to the question above in the higher-dimensional Gorenstein case.

\begin{thm}\label{high-good}
Suppose that $A$ is a Gorenstein local ring of $d=\dim A \ge 3$
and that $I \subset A$ is a good ideal.
Set $\calR=\calR(I)$. Then we have the following.
\begin{enumerate}
 \item[$(1)$] The following conditions are equivalent:
 \begin{enumerate}
  \item[$(a)$]  $\calR$ is an almost Gorenstein graded ring;
  \item[$(b)$] $\calR$ is Gorenstein;
  \item[$(c)$] $d=3$.
 \end{enumerate}
 \item[$(2)$]  If $A$ is a regular local ring, then $\calR_{\fkM}$ is
 an almost Gorenstein local ring. Conversely, if $\calR_{\fkM}$ is an
 almost Gorenstein local ring but not Gorenstein,
 then $A$ itself is a regular local ring.
\end{enumerate}
\end{thm}

\begin{proof}
(1) $(c)\Longrightarrow(b)$
By \cite[Proposition 2.2]{GIW}, $G(I)$ is a Gorenstein ring with $\rma(G(I))=1-d=-2$.
Hence, $\calR(I)$ is Gorenstein by the Goto--Shimoda theorem (see \cite{GS}).

\medskip

\noindent
$(b)\Longrightarrow(a)$ This is trivial.

\medskip

\noindent
$(a)\Longrightarrow(c)$ Now suppose that
$\calR$ is an almost Gorenstein graded ring but not Gorenstein.
Then there exists an exact sequence of graded
$\calR$-modules:
\[
0 \to \calR \stackrel{\varphi}{\longrightarrow}   \K_{\calR}(1) \to C \to 0
\]
such that $\mu_{\calR}(C)=\rme_{\fkM}^0(C)$.
As $\varphi(1) \notin \fkM \K_{\calR}(1)$, we may assume
$\varphi(1) =t$. Then
\[
C=A/I \oplus A/I^2 \oplus \cdots\oplus A/I^{d-2} \oplus I/I^{d-1} \oplus \cdots \supset \calR/I^{d-2}\calR(-d+2)
\]
and $\dim C=\dim \calR/I^{d-2}\calR \ (=d)$.
This implies that $\rme_{\fkM}^0(C) \ge \rme_{\fkM}^0(\calR/I^{d-2}\calR)$.
For each prime $P \in \Assh(\calR/I\calR)$, we have
$\dim \calR_P=\height I\calR=1$ and
$I^{d-2}\calR_P \subsetneq \cdots \subsetneq I\calR_P \subsetneq \calR_P$.
Thus, the associative formula implies that
\[
\rme_{\fkM}^0(C) \ge \rme_{\fkM}^0(\calR/I^{d-2}\calR)
\ge \ell_{\calR_P}(\calR_P/I^{d-2}\calR_P) \ge d-2.
\]
On the other hand,
\[
\rme_{\fkM}^0(C) = \mu_{\calR}(C)=\mu_{\calR}(\K_{\calR})-1=(d-2)-1=d-3.
\]
This is a contradiction.

\medskip

(2) First suppose that $A$ is a regular local ring.
Since $I$ is a good ideal, we have that
$\calR$ is a Cohen-Macaulay ring with
$\K_{\calR} \cong \calR t + \calR t^2+\cdots+\calR t^{d-2}$.
Hence, $\calR_{\fkM}$ is an almost Gorenstein local ring by Lemma
\ref{useful}.

Conversely, assume that $\calR_{\fkM}$ is an almost Gorenstein
local ring but not Gorenstein.
We may assume that $d \ge 4$.
Take an exact sequence
\[
0 \to \calR_{\fkM} \to \K_{\calR_{\fkM}} \to C_{\fkM} \to 0
\]
so that $C_{\fkM}$ is an Ulrich $\calR_{\fkM}$-module.
Then we obtain
\[
\mu_{\calR}(\fkM \K_{\calR})
\le \mu_{\calR} (\fkM)+\mu_{\calR}(\fkM C)
\le \mu_{\calR}(\fkM)+d \cdot (\mu_{\calR}(\K_{\calR})-1).
\]
Since $\mu_{\calR}(\K_{\calR}) =d-2$,  $\mu_{\calR}(\fkM)=\mu_A(\m)+\mu_A(I)$ and
\[
\fkM \cdot \K_{\calR}=(\m t,\m t^2, \ldots, \m t^{d-2}, I t^{d-1}),
\]
we obtain
\[
(d-2)\mu_A(\m)+\mu_A(I) \le \mu_A(\m)+\mu_A(I)+d(d-3).
\]
This implies $\mu_A(\m)=d$.
That is, $A$ is a regular local ring, as required.
\end{proof}

\section{Higher-dimensional case
 (Proof of Theorem \ref{Main-RLRpower})}

In this section, we prove Theorem \ref{Main-RLRpower}.
In what follows, let $(A,\m)$ be a regular local ring of dimension $d \ge 2$
with infinite residue class field, and let $\ell \ge 1$ be an integer.

\begin{rem}
First suppose $\ell=1$.
Then $\calR(\m^{\ell}) = \calR(\m)$ is an almost Gorenstein graded ring because the maximal ideals
$\m$ of a regular local ring is a parameter ideal;
see \cite[Theorem 1.3]{GMTY1}.

Next suppose that $d=2$.
Then \cite[Corollary 1.4]{GMTY2} implies that $\calR(\m^{\ell})$ is
an almost Gorenstein graded ring for every $\ell \ge 1$.

Finally, suppose that $\ell =d-1$.
Then $\calR(\m^{d-1})$ is a Gorenstein ring and, thus, it is an
almost Gorenstein graded ring;
see e.g.\ \cite[Proposition 2.3]{GMTY1}.
\end{rem}

Thus, we restrict our attention to the case where $\ell \ge 2$ and
$d \ge 3$ to prove Theorem \ref{Main-RLRpower}.

\begin{prop} \label{AG-local}
Let $\ell \ge 2$ and $d \ge 3$ be integers.
Assume that $\mathcal{R}(\fkm^{\ell})_{\mathfrak{M}}$ is an almost
Gorenstein local ring. Then $\ell$ is a divisor of $d-1$.
\end{prop}

\begin{proof}
The graded canonical module of the Rees algebra
$\mathcal{R}(\m)$ is given  by
\[
\K_{\mathcal{R}(\m)} \cong
At + At^2 + \cdots + At^{d-2} +
\sum_{n \ge d-1} \m^{n-d+1} t^n;
\]
see e.g.\ \cite[Lemma 5.1]{GI}.
This formula and \cite[Proposition 2.5]{HHK} imply
\[
\K_{\mathcal{R}(\m^{\ell})} \cong
\sum_{n=1}^b At^n + \sum_{n\ge b+1}
\m^{n\ell-d+1}t^n, \quad
\text{where}\ b= \big\lfloor \textstyle{\frac{d-2}{\ell}} \big\rfloor=
\big\lceil \textstyle{\frac{d-1}{\ell}} \big\rceil -1.
\]

Now, set $J=\m^{(b+1)\ell-d+1}$,
$\mathcal{R}=\mathcal{R}(\m^{\ell})$ and
$\mathfrak{M}=\mathfrak{m}\mathcal{R}+\mathcal{R}_{+}$.
Then $\K:=\K_{\mathcal{R}(\m^{\ell})}$ is generated by
$t,t^2,\ldots,t^{b}$ and $Jt^{b+1}$ as an $\mathcal{R}$-module.
Hence,
\begin{equation} \label{eq:genK}
\mu_{\mathcal{R}}(\K)=b+\mu_A(J).
\end{equation}

Similarly, since $\mathfrak{M}\K$ is generated by
$\mathfrak{m}t, \mathfrak{m}t^2,\ldots,\mathfrak{m}t^{b}$,
$\mathfrak{m}Jt^{b+1}$ and $\mathfrak{m}^{\ell}Jt^{b+2}$,
we have
\begin{equation} \label{eq:genMK}
\mu_{\mathcal{R}}
(\mathfrak{M}\K)=b \cdot
\mu_A(\mathfrak{m})+\mu_A(\mathfrak{m}J) +
\mu_A(\mathfrak{\m}^{\ell}J).
\end{equation}

Now assume that $\mathcal{R}_{\mathfrak{M}}$
is an almost Gorenstein local ring.  Then we must prove the following claim.

\medskip

\begin{flushleft}
{\it Claim.} We claim that
$\mu_A(\m J) + \mu_A(\m^{\ell}J) \le \mu_A(\m^{\ell})+d \cdot \mu_A(J)$.
\end{flushleft}

\medskip

Now we prove the claim.
We consider the exact sequence
\[
0 \to \calR_{\fkM} \to \rmK_{\calR_\fkM} \to C \to 0
\]
of $\calR_{\fkM}$-modules with $\mu_{\calR_{\fkM}}(C) =
\e^0_{\fkM\calR_{\fkM}}(C)$.
Then \cite[Corollary 3.10]{GTT} implies that
\[
\mu_{\mathcal{R}_{\fkM}}(C)=\mu_{\mathcal{R}_{\fkM}}(\K_{\calR_{\fkM}})-1 = \mu_{\calR}(\rmK_{\calR})-1
\]
and
\[
0 \to \mathfrak{M}\calR_{\fkM} \to \mathfrak{M} \K_{\calR_{\fkM}} \to \mathfrak{M} C \to 0
\]
is exact.
Moreover, as $C$ is an Ulrich $\mathcal{R}_{\fkM}$-module and
$A$ is regular,  we have
\begin{eqnarray*}
\mu_{\mathcal{R}}(\mathfrak{M}\K) = \mu_{\calR_{\fkM}}(\mathfrak{M} \K_{\calR_{\fkM}})
&\le &
\mu_{\mathcal{R}_{\fkM}}(\mathfrak{M}\calR_{\fkM}) +
\mu_{\mathcal{R}_{\fkM}}(\mathfrak{M}C) \\
& \le & \mu_{\mathcal{R}}(\mathfrak{M}) +
d \cdot \mu_{\mathcal{R}_{\fkM}}(C) \\
& = & \mu_A(\m) + \mu_A(\m^{\ell})+d \cdot (\mu_{\mathcal{R}}(\K)-1) \\
&=& \mu_A(\m^{\ell}) +
d \cdot \mu_{\mathcal{R}}(\K).
\end{eqnarray*}


By substituting $(\ref{eq:genK})$ and $(\ref{eq:genMK})$ for this,
we obtain the desired claim.

Note that $\mu_A(\m^k)={k+d-1 \choose d-1}$ and
$J=\m^{(b+1)\ell-d+1}$.
Using the above claim, we have
\begin{equation}
{(b+1)\ell+1 \choose d-1} + {(b+2)\ell \choose d-1}
\le
{\ell+d-1 \choose d-1} + d \cdot {(b+1)\ell \choose d-1}.
\end{equation}
The opposite inequality follows from the following lemma and, thus,
$\ell$ is a divisor of $d-1$.
\end{proof}

\begin{lem} \label{Ineq}
Let $\ell \ge 2$ and $d \ge 3$ be integers and set $b= \big\lfloor \textstyle{\frac{d-2}{\ell}} \big\rfloor$. Then
\begin{equation} \label{eq:OppIneq}
{(b+1)\ell+1 \choose d-1} + {(b+2)\ell \choose d-1}
\ge
{\ell+d-1 \choose d-1} + d \cdot {(b+1)\ell \choose d-1}
\end{equation}
holds true.
In addition, equality holds true if and only if $\ell$ is a divisor of $d-1$.
\end{lem}

\begin{proof}
Set $i=d-2-b\ell$. Then $0 \le i \le \ell-1$. Note that
\[
{n \choose m} = {n-1 \choose m} + {n-1 \choose m-1} \quad
\text{if $n, m \ge 2$}.
\]
By substituting the above equality, we have
\begin{eqnarray*}
{(b+2)\ell \choose d-1}
&=&
{(b+1)\ell+\ell-1 \choose d-1}  + {(b+1)\ell+\ell-1 \choose d-2}\\
& = & {(b+1)\ell+\ell-2 \choose d-1}
 + {(b+1)\ell+\ell-2 \choose d-2}
+ {(b+1)\ell + \ell-1 \choose d-2} \\
&=& \cdots = {(b+1)\ell+i+1 \choose d-1} +
\sum_{j=i+1}^{\ell-1} {(b+1)\ell+j \choose d-2} \\
&=& {\ell + d-1 \choose d-1} + \sum_{j=i+1}^{\ell-1} {(b+1)\ell+j \choose d-2}.
\end{eqnarray*}
On the other hand,
\begin{eqnarray*}
d \cdot {(b+1)\ell \choose d-1}
&=& {(b+1)\ell \choose d-1} + (d-1){(b+1)\ell \choose d-1} \\
&=& {(b+1)\ell \choose d-1} + \big((b+1)\ell-d+2\big){(b+1)\ell \choose d-2} \\
&=& {(b+1)\ell \choose d-1} +(\ell-i) {(b+1)\ell \choose d-2} \\
&=& {(b+1)\ell+1 \choose d-1} + (\ell-i-1) {(b+1)\ell \choose d-2}.
\end{eqnarray*}
Thus,
\begin{eqnarray*}
&& {(b+1)\ell+1 \choose d-1} + {(b+2)\ell \choose d-1} - \left\{
{\ell+d-1 \choose d-1} + d \cdot {(b+1)\ell \choose d-1}\right\} \\
&& \\
&& \ \ =
\sum_{j=i+1}^{\ell-1} \left\{ {(b+1)\ell+j \choose d-2} -  {(b+1)\ell \choose d-2} \right\} \ge 0
\end{eqnarray*}
because $i \le \ell-1$.
Therefore, the equality of $(\ref{eq:OppIneq})$ holds if and only if  $i=\ell -1$.
\end{proof}

To complete the proof of
Theorem \ref{Main-RLRpower},
we give the following theorem.

\begin{thm} \label{NotGraded}
Suppose that $\ell \ge 2$, $d \ge 3$, and $\ell\mid  d-1$.
Set $\calR=\calR(\m^{\ell})$.
Then:
\begin{enumerate}
\item[$(1)$] $\calR_{\fkM}$ is an almost Gorenstein local ring;
\item[$(2)$] if $\ell \ne d-1$, then
$\calR$ is {\em not} an almost Gorenstein graded ring.
\end{enumerate}
\end{thm}

\begin{proof}
Set $b=\frac{d-1}{\ell}-1$.
Then $b \ge 0$ and
\[
\K_{\mathcal{R}} \cong \calR t + \calR t^2 + \cdots
\ + \calR t^b + \mathcal{R}t^{b+1}.
\]

\medskip

\noindent
 (1) This follows from Lemma \ref{useful}.

\medskip

\noindent
(2)
Now suppose that $\mathcal{R}$ is an almost Gorenstein graded ring.
Then there exists a short exact sequence
\[
0 \to \mathcal{R} \stackrel{\varphi}{\to} \K_{\mathcal{R}}(1) \to C \to 0
\]
of graded $\mathcal{R}$-modules so that $C$ is an Ulrich $\mathcal{R}$-module.
Since $\varphi(1)$ is part of a minimal set of generators of $[\K_{\mathcal{R}}]_1$
by \cite[Corollary 3.10]{GTT}, we may assume that $\varphi(1)=t$
without loss of generality.
Then
\[
C=\K_{\mathcal{R}}(1)/\mathcal{R}t \cong \mathcal{R}\overline{t^2} + \cdots + \mathcal{R}\overline{t^{b+1}}
\]
yields that $\mu_{\mathcal{R}}(C) \le b < \frac{d-1}{\ell}$.

If we put $I=\fkm^{\ell}$, then
\[
C =\sum_{n=2}^{\infty}C_n \cong
A/I \oplus A/I^2 \oplus \cdots \oplus A/I^b \oplus I/I^{b+1} \oplus \cdots.
\]
Thus,
\[
C \supset \mathcal{R}\overline{t^{b+1}} \cong \mathcal{R}/I^{b}\mathcal{R}(-(b+1)),
\]
and, hence,
\[
\rme_{\mathfrak{M}}^0(C)
\ge \rme_{\mathfrak{M}}^0(\mathcal{R}/I^{b}\mathcal{R})
\ge \rme_{\mathfrak{M}}^0(\mathcal{R}/I \mathcal{R})
=\rme_{\mathfrak{M}}^0(G(I))
=\rme_{I}^0(A)=\ell^d \cdot \rme_{\fkm}^0(A)=\ell^d.
\]
As $\ell \ge 2$ and $b \ge 1$, we have $\ell^d > \frac{d}{\ell} > \frac{d-1}{\ell}=b+1$.
Therefore, $\rme_{\mathfrak{M}}^0(C) > b+1 \ge \mu_{\mathcal{R}}(C)$,
which contradicts the assumption.
\end{proof}

In Table~\ref{tab1} we present part of the list of
$(d,\ell)$ for which, over a $d$-dimensional regular local ring $(A, \m)$, the Rees algebra $\calR(\m^{\ell})$ is a Gorenstein ring ({\bf Gor}),
an almost Gorenstein graded ring ({\bf AG}),
or $\calR(\m^{\ell})_{\fkM}$ is an almost Gorenstein local
ring ({\bf AGL}).

\begin{table}[htb]
\begin{center}
\caption{When is $\calR(\m^{\ell})$ almost Gorenstein?\label{tab1}}
{\small
\begin{tabular}{|c||c|c|c|c|c|c|c|c|c|} \hline
$d\setminus \ell$ & 1 & 2 & 3 & 4 & 5 & 6 & 7 & 8 & 9 \\ \hline \hline
2 & {\bf Gor}& {\bf AG} & {\bf AG} & {\bf AG} & {\bf AG} & {\bf AG} & {\bf AG} & {\bf AG}& {\bf AG} \\ \hline
3 & {\bf AG} & {\bf Gor} &X & X & X & X & X & X & X \\ \hline
4 & {\bf AG} &  X   &{\bf Gor}& X & X & X & X & X & X \\ \hline
5 & {\bf AG} & {\bf AGL} &X & {\bf Gor}& X & X & X & X & X \\ \hline
6 & {\bf AG} & X  &X & X & {\bf Gor}& X & X & X & X \\ \hline
7 & {\bf AG} & {\bf AGL}    &  {\bf AGL}& X & X & {\bf Gor}& X & X & X \\ \hline
8 & {\bf AG} &  X &X & X & X & X & {\bf Gor}& X & X \\ \hline
9 & {\bf AG} &  {\bf AGL}    &X &  {\bf AGL}& X & X & X & {\bf Gor}& X \\ \hline
10& {\bf AG} & X  &  {\bf AGL}  & X & X & X & X & X & {\bf Gor}\\ \hline
\end{tabular}
}
\end{center}
\end{table}

\begin{acknowledgement}
  We thank the referee for a careful reading of the original manuscript.
We would like to thank Editage (www.editage.jp) for English language editing.
The first author was partially supported by JSPS Grant-in-Aid for Scientific Research 16K05112.
The second author was partially supported by JSPS Grant-in-Aid for Scientific Research 26400054. The third author was partially supported by JSPS Grant-in-Aid for Young Scientists (B) 17K14176. The fourth author was partially supported by JSPS Grant-in-Aid for Scientific Research 16K05110.
\end{acknowledgement}


\end{document}